\documentclass[%
 reprint,
 amsmath,amssymb,
 aps,
]{revtex4-2}

\usepackage{graphicx}% Include figure files
\usepackage{dcolumn}% Align table columns on decimal point
\usepackage{bm}% bold math
\usepackage{bbm}
\usepackage{amsthm}
\usepackage{bbm}
\newtheorem{theorem}{Theorem}
\newtheorem{lemma}{Lemma}
\newtheorem{prop}{Proposition}
\newtheorem{corollary}{Corollary}
\newtheorem{remark}{Remark}
\newtheorem{definition}{Definition}

\usepackage{dsfont}
\usepackage{mathtools}
\usepackage{soul}
\usepackage{xcolor}
\let\ACMmaketitle=\maketitle
\renewcommand{\maketitle}{\begingroup\let\footnote=\thanks \ACMmaketitle\endgroup}
\newcommand{\YZ}[1]{\textcolor{black}{#1}}

\DeclarePairedDelimiterX{\inp}[2]{\langle}{\rangle}{#1, #2}
\begin{document}

\preprint{APS/123-QED}

\title{Eigenvalues of Autocovariance Matrix: A Practical Method to Identify the Koopman Eigenfrequencies}% Force line breaks with \\

\author{Yicun Zhen}
  \email{zhenyicun@protonmail.com}

\author{Bertrand Chapron}%
\affiliation{%
Institut Français de Recherche pour l'Exploitation de la Mer, Plouzan{\'e}, France
}%


\author{Etienne M\'{e}min}
\affiliation{
INRIA/IRMAR Campus universitaire de Beaulieu, Rennes, 35042 Cedex, France}
\altaffiliation{IRMAR - Institut de Recherche Math{\'e}matique de Rennes \\INRIA - Institut national de recherche en informatique et en automatique
}%

\author{Lin Peng}
\affiliation{%
Ocean University of China, Qingdao, China
}%
\altaffiliation{Also at Jiangsu Ocean University, Lianyungang, China}

\date{\today}% It is always \today, today,

\begin{abstract}
     To infer eigenvalues of the infinite-dimensional Koopman operator, we study the leading eigenvalues of the autocovariance matrix associated with a given observable of a dynamical system. 
     For any observable $f$ for which all the time-delayed autocovariance exist, we construct a Hilbert space $\mathcal{H}_f$ and a Koopman-like operator  $\mathcal{K}$ that acts on $\mathcal{H}_f$. We prove that the leading eigenvalues of the autocovariance matrix has one-to-one correspondence with the energy of $f$ that is represented by the eigenvectors of $\mathcal{K}$. The proof is associated to several representation theorems of isometric operators on a Hilbert space, and the weak-mixing property of the observables represented by the continuous spectrum. We also provide an alternative proof of the weakly mixing property. When $f$ is an observable of an ergodic dynamical system which has a finite invariant measure $\mu$, $\mathcal{H}_f$ coincides with closure in $L^2(X,d\mu)$ of Krylov subspace generated by $f$, and $\mathcal{K}$ coincides with the classical Koopman operator. The main theorem sheds light to the theoretical foundation of several semi-empirical methods, including singular spectrum analysis (SSA), data-adaptive harmonic analysis (DAHD), Hankel DMD and Hankel alternative view of Koopman analysis (HAVOK). It shows that, when the system is ergodic and has finite invariant measure, the leading temporal empirical orthogonal functions indeed correspond to the Koopman eigenfrequencies.  A theorem-based practical methodology is then proposed to identify the  eigenfrequencies of $\mathcal{K}$ from a given time series. It builds on the fact that the convergence of the renormalized eigenvalues of the Gram matrix is a necessary and sufficient condition for the existence of $\mathcal{K}-$eigenfrequencies. Numerical illustrating results on  simple low dimensional systems and real interpolated ocean sea-surface height data are presented and discussed.
\end{abstract}

\maketitle

%\tableofcontents

\section{Introduction}
%The dynamic mode decomposition (DMD) algorithm, initially described by \cite{Saad1980}, is a powerful data-driven approach, largely popularized by \cite{Schmid2010}. The DMD is likely ideally suited to analyze complex high-dimensional geophysical flows. 
The dynamic mode decomposition (DMD) algorithm, is a powerful and versatile data-driven approach proposed by \cite{Schmid2010}, ideally suited to analyze complex high-dimensional geophysical flows in terms of recurrent or quasi-periodic modes. The DMD is indeed related to the Koopman theory \cite{Rowley2009}, stating that observables of an Hamiltonian system can always be described via a linear  transformation. The original DMD algorithm has a lot of common points with the algorithm presented in \cite{Saad1980}. For practical applications and real data analysis, several follow-up algorithms have been proposed. To name a few, it can be listed the optimized DMD \cite{Chen2012}, the optimal mode decomposition \cite{wynn2013}, the exact DMD\cite{Tu2014}, the Hankel DMD \cite{Arbabi2017}, the sparsity promoting DMD \cite{Kusaba2020}, the multi-resolution DMD \cite{Kutz2015}, the extended DMD \cite{Williams2015}, DMD with control \cite{Proctor2016}, total least squares DMD \cite{Hemati2017}, dynamic distribution decomposition \cite{TaylorKing2020}, etc. These DMD algorithms are generally motivated by different reasons, but a key overall objective is to help provide the most precise numerical approximation of the Koopman operator. 
When the system is ergodic and measure-preserving, it would indeed be equivalent to have a precise description of the spectrum $\nu$ of Koopman operator restricted on $\mathcal{H}_f$ and a precise mapping between $\mathcal{H}_f$ and $L^2(S^1,d\nu)$ (where $\mathcal{H}_f$ is the linear subspace generated by a single observable $f$ and $S^1$ the unit complex circle, see section 2 for detailed definition of these spaces).  
The authors of \cite{Korda2018} proved the convergence in the strong operator topology of extended DMD algorithm, provided a complete orthogonal basis of the space of square-integrable observables. In \cite{Arbabi2017}  the convergence of Hankel DMD algorithm is proved for the finite dimensional case, which corresponds to the finite truncation of the discrete part of the spectrum.  Christoffel-Darboux kernel is exploited in \cite{KORDA2020599} to directly identify the discrete component and the absolutely continuous component of the spectrum.
Note, DMD algorithms are not the only way to approximate Koopman operator. In a series of papers ( \cite{WilliamsRowleyKevrekidis2015},\cite{Das2017}, \cite{Das2018},\cite{Giannakis2018} and \cite{Giannakis2020}), the approximation of Koopman operator is performed by kernel methods. Recently, \cite{Giannakis2020} showed the convergence of kernel methods for any measure preserving ergodic dynamical systems, the measure of which support lies on a compact  manifold. A related issue for stochastic dynamics can be found in \cite{Nuske2021}

In this manuscript, we define a Hilbert space $\mathcal{H}_f$ and the associated Koopman-like time-shift operator $\mathcal{K}$ for each given time series the autocovariances of which exist. When the time series is given by an observable $f$ on a ergodic dynamical system with finite invariant measure, $\mathcal{H}_f$ coincides with the closure of the infinite dimensional Krylov subspace generated by the classical Koopman operator and $f$, and $\mathcal{K}$ coincides with the classical Koopman operator on $\mathcal{H}_f$. We argue and prove that for any observable $f$ such that all the time-delayed autocovariance exist, when the  parameters of the autocovariance matrix $C_{NM}(f)$ goes to infinity in the right order, the leading eigenvalues renormalized by the dimension of $C_{NM}(f)$ converge to the energy of $f$ that is represented by the  eigenvector of $\mathcal{K}$. All other renormalized eigenvalues shall further converge to 0 uniformly. Despite its theoretical interests, the main theorem directly suggests a practical algorithm to explicitly identify the Koopman eigenfrequencies together with the associated energy from given time series. As a by-product, it also shows that the leading temporal empirical orthogonal functions calculated by singular spectrum analysis (SSA, \cite{Ghil2002AdvancedSM}) method are indeed represented by the eigenfrequencies of $\mathcal{K}$. Similarly, this theorem also sheds light on the theoretical foundation of data-adaptive harmonic decomposition (DAHD,\cite{Kondrashov2020DataadaptiveHA}, and Hankel alternative view of Koopman analysis (HAVOK, \cite{Brunton2017ChaosAA}). Because all these methods are based either on trajectory matrix or on $C_{NM}(f)$.
 
The paper is organized as follows. In section 2, we present our main result and the necessary mathematical background knowledge. We also discuss about how the main theorem provides theoretical support to SSA,DAHD, and HAVOK. In section 3, we state the details of the proof of the main result. In section 4, we present the detailed algorithm and compare it with another numerical method based on Yosida's mean ergodic theorem (\cite{Yosida1995}). In section 5 we present  numerical results on two simple low dimensional measure preserving ergodic dynamical systems and interpolated ocean sea-surface height data. Section 6 concludes this study and gives some perspectives. The necessary code and data that reproduces all the numerical results can be accessed at https://doi.org/10.5281/zenodo.5585970.

\section{Preliminaries and the main result}
Given a continuous-time dynamical system 
\[
X_t = \Phi_t(X_0), 
\]
and an observable $f(X_t)$, we have a time series $\{f(X_t): t=0,\Delta t,2\Delta t,...\}$. We assume that the time-delayed autocovariance of $f$ exists for all $l\in\mathbb{N}$:
\begin{align}
    \rho_{l\Delta t} = \lim_{n\to\infty}\frac{1}{n}\sum_{k=0}^nf(X_{k\Delta t})\bar{f}(X_{(k+l)\Delta t}).
\end{align}
To avoid misinterpretation, we use $F_0 = \{f(X_0), f(X_{\Delta t}),...\}$ to denote time series associated to $f$ and a given (fixed) orbit of the dynamical system. For $l\in\mathbb{N}$, we define $F_l$ to be the time shifted time series $\{f(X_{l\Delta t}), f(X_{(l+1)\Delta t}),...\}$. For any $a,b\in\mathbb{C}$ and any two time series, associated to the same dynamics 
\begin{align}
    G = \{g_0,g_1,...\}\nonumber \\
    H = \{h_0,h_1,...\}\nonumber,
\end{align}
we define
\begin{align}
    aG + bH = \{ag_0+bh_0, ag_1+bh_1, ...\}. \nonumber 
\end{align}
Let 
\begin{align}
    \tilde{H}_f = \{\sum_{i=1}^n c_i F_{l_i}: n\geq 1, c_i\in\mathbb{C}, l_i\in\mathbb{N}\}.
\end{align}
Then $\tilde{H}_f$ is a linear space. Now we define the Koopman-like (or time shift) operator on $\tilde{H}_f$. We start with for any $l\in\mathbb{N}$
\begin{align}
    \mathcal{K}^{\Delta t}F_l = F_{l+1}, 
\end{align}
and then generalize the action of $\mathcal{K}^{\Delta t}$ to the whole $\tilde{H}_f$. It is not hard to show that $\mathcal{K}^{\Delta t}: \tilde{H}_f\to \tilde{H}_f$ is well-defined.

The existence of $\rho_s$ allows us to define an inner product on  $\tilde{\mathcal{H}}_f$ by
\begin{align}
    \inp{h}{g}  = \lim_{n\to\infty}\frac{1}{n}\sum_{k=0}^{n-1}h_k\bar{g}_k,
\end{align}
where $h,g\in\tilde{\mathcal{H}}_f$. 
For any $l_1,l_2\in\mathbb{N}$, it is obvious that 
\begin{align}
    \rho_{l_1\Delta t} = \lim_{n\to\infty}\frac{1}{n}\sum_{k=0}^nf(X_{(n+l_2)\Delta t})\bar{f}(X_{(n+l_1+l_2)\Delta t}).
\end{align}
Hence $\mathcal{K}^{\Delta t}$ preserves the inner product in $\tilde{\mathcal{H}}_f$ and is hence continuous.

Let $\mathcal{H}_f$ be the completion of $\tilde{\mathcal{H}}_f$. $\mathcal{H}_f$ is a Hilbert space. $\mathcal{K}^{\Delta t}$ preserves the inner product of $\tilde{\mathcal{H}}_f$.  Therefore $\mathcal{K}^{\Delta t}$ can be extended, by continuity, to an isometric operator that acts on $\mathcal{H}_f$. For sake of simplicity, keeping the same notation for the extended operator, $\mathcal{K}^{\Delta t}$, whose domain is $\mathcal{H}_f$, is the Koopman-like (or time shift) operator we study in this paper.

\begin{remark}
The classical Koopman operator is defined to act on some function space on the whole phase space of some dynamical system, i.e.
\begin{align}
    \mathcal{K}^{\text{cl}}g (x) = g(\Phi(x)),
\end{align}
where $x\in X$ is a state and $\Phi$ the discrete-time flow of the dynamical system.

In our setting, the definition of $\mathcal{K}^{\Delta t}$ and $\mathcal{H}_f$ purely relies on the time series. The dynamical system is hidden behind. It is possible that the time series only reveals partial properties of the dynamical system.  When the discrete-time dynamical system is ergodic and has finite invariant measure $\mu$, Birkhoff ergodic theorem guarantees that $\mathcal{H}_f$ is isomorphic to the closure of Krylov subspace $\overline{\text{Span}\{f,\mathcal{K}^{\text{cl}}f,...\}}\subset L^2(X,d\mu)$ as Hilbert spaces. The isomorphism is given by 
\begin{align}
    \phi: \sum_{i=1}^nc_iF_{l_i} \rightarrow \sum_{i=1}^nc_i(\mathcal{K}^{\text{cl}})^{l_i}f    
\end{align}
And it is not hard to see that $\mathcal{K}^{\Delta t}$ acts on $\mathcal{H}_f$ in a similar way as $\mathcal{K}^{\text{cl}}$ on the space of observables $L^2(X,d\mu)$, i.e. for any $h\in \mathcal{H}_f$,
\begin{align}
    \phi(\mathcal{K}h) = \mathcal{K}^{\text{cl}}(\phi(h)).
\end{align}
The elements in $\tilde{\mathcal{H}}_f$  can always be represented as some time series. But in general we can not assert that any element in $\mathcal{H}_f$ can be represented as some time series. In particular, we do not know if the eigenvectors $v_i\in\mathcal{H}_f$ of $\mathcal{K}$ can be represented as a time series in the form $\{1,\xi_i, \xi_i^2,...\}$. Nor in general can we identify $\mathcal{H}_f$ with some space of functions on the phase space. Hence, without these assumptions, the time-shift operator and the Koopman operator cannot be strictly related.  Ergodicity+finite invariant measure is a stronger assumption than the existence of autocovariance. As long as the autocovariances exist, the quantities mentioned in the main result are mathematically well-defined. Therefore we do not restrict ourselves to the case where the system is ergodic and has finite invariant measure. But we keep the generality in the definition of $\mathcal{H}_f$ and $\mathcal{K}$. The physical meaning of $\mathcal{H}_f$ and $\mathcal{K}$ in the general case needs to be studied further. For readers who are interested in the classical Koopman operator $\mathcal{K}^{\text{cl}}$, we point out that $\mathcal{K}$ indeed coincides with $\mathcal{K}^{cl}$ when the system is ergodic and has finite invariant measure. 
\end{remark}

If $\{f(X_t): t\geq 0\}$ is a continuous-time time series, we assume that  
\begin{align}
    \rho_s = \lim_{T\to\infty}\frac{1}{T}\int_{0}^Tf(X_t)\bar{f}(X_{t+s})dt
\end{align}
exists for every $s\geq 0$. Similar to the discrete-time case, we define 
\begin{align}
    F_{s} = \{f(X_t):t\geq s\}    
\end{align}
to be the continuous time series that starts from $f(X_s)$.
We can define the linear space
\begin{align}
\tilde{\mathcal{H}}^{\text{cont}}_f = \{\displaystyle\sum_{i=1}^n c_iF_{s_i}: c_i\in\mathbb{C}, s_i\geq 0, n\geq 1\},
\end{align}
with inner product:
\begin{align}
    \inp{h}{g} = \lim_{T\to\infty}\frac{1}{T}\int_{0}^Th(t)\bar{g}(t)dt.
\end{align}
Let $\mathcal{H}^{\text{cont}}_f$ be the completion of $\tilde{\mathcal{H}}^{\text{cont}}_f$. $\mathcal{K}^s$ is an isometry on $\mathcal{H}^{\text{cont}}_f$ for all $s\geq 0$.

We are interested in the eigenvalues and eigenvectors of $\mathcal{K}^{\Delta t}$ (for discrete-time case) and $\mathcal{K}^s$ (for continuous-time case). A natural question to ask is that whether the eigenfrequencies of $\mathcal{K}^{\Delta t}$ and $\mathcal{K}^{s}$ are the same. We have the following result.
\begin{prop}\label{prop:cont K = disc K}
Let $\{f(X_t):t\geq 0\}$ be a continuous time process for which $\rho_s$ exists for all $s\geq 0$. Assume that the curve $\mathcal{K}^s: [0,\infty)\to\mathcal{H}_f^{\text{cont}}$ is continuous in $s$. Let $\Delta t>0$ be a time step. Assume that 
\begin{align}
    &\lim_{T\to\infty}\frac{1}{T}\int_{0}^Tf(X_t)\bar{f}(X_{t+s})dt\nonumber \\
    =&\lim_{T\to\infty}\frac{\Delta t}{T}\sum_{\mathbb{N}\ni n=0}^{T/\Delta t}f(X_{n\Delta t})\bar{f}(X_{n\Delta t+s}),\label{eq: inner product cont=disc}
\end{align}
then $\mathcal{H}_f \hookrightarrow \mathcal{H}^{\text{cont}}_f$.
Let $q$ be an eigenfrequency of the discrete-time operator $\mathcal{K}^{\Delta t}$, i.e. there exists $h\in \mathcal{H}_f \hookrightarrow \mathcal{H}^{\text{cont}}_f$ so that $\mathcal{K}^{\Delta t}h = e^{iq}h$. Then there exists an integer $k$, and $h_k\in\mathcal{H}^{\text{cont}}_f$, so that
\begin{align}
    \mathcal{K}^sh_k = e^{i\frac{q+2k\pi}{\Delta t}s}h_k
\end{align}
for all $s\geq 0$.
\end{prop}
The proof of this proposition is given in the appendix. Proposition \ref{prop:cont K = disc K} guarantees that for every eigenfrequency $q$ of $\mathcal{K}^{\Delta t}$, there always exists an eigenfrequency of $\mathcal{K}^s$ for all $s\geq 0$ which reduces to $q$ at discrete time step.

Our discussion about continuous-time time series stops here. Hereinafter we always assume that the time series $\{f(X_t)\}$ is discrete in time. For simplicity the time series is denoted by $\{f(0), f(1),...\}$ and  we use $f$ to denote the whole time series generated by $f$ as an element in $\mathcal{H}_f$.  We use the notation $\mathcal{K}$ to replace $\mathcal{K}^{\Delta t}$. Recall that the time-shift operator $\mathcal{K}$ is an isometry on the Hilbert space $\mathcal{H}_f$ of time-series associated to observable $f$.

\begin{remark}
Due to Birkhoff ergodic theorem, autocovariance exists when the dynamical system is ergodic and has a finite invariant measure. 
\end{remark}

The following theorem provides a very useful result to decompose any isometric operator and the Hilbert space on which it acts into unitary and non unitary part. 
\begin{theorem}[Wold decomposition]\label{thm:Wold}
Let $\mathcal{H}$ be a Hilbert space and $\mathcal{V}$ any isometry of $\mathcal{H}$. Then we have an orthogonal decomposition $\mathcal{H} = \mathcal{H}_{NU}\bigoplus\mathcal{H}_U$, and $\mathcal{V} = \mathcal{V}_{NU}\bigoplus\mathcal{V}_U$, such that $\mathcal{H}_{NU} = \displaystyle\bigoplus_{s\in I}\mathcal{H}_{s}$,  $\mathcal{V}_{NU}$ acts on $\mathcal{H}_s$ for some index set $I$ as a unilateral shift, i.e. $\mathcal{V}_{NU}(v_0,v_1,...) = (0, v_0,v_1,...)$. And $\mathcal{V}_U$ acts on $\mathcal{H}_U$ and is unitary. $\mathcal{H}_{NU}$ is called the completely non unitary part of $\mathcal{H}$ as it does not contain  closed subspaces of $\mathcal{H}$ on which $\mathcal{V}$ acts as a unitary operator.   
\end{theorem}
Wold theorem is a particular case of  (Sz\"{o}kefalvi-Nagy–Foia\;{s}) theorem for contraction operator. 
\begin{theorem}[Sz\"{o}kefalvi-Nagy–Foias]
Let $\mathcal{T}$ be a contraction operator (i.e. $\|\mathcal{T}\| \leq 1$) on a Hilbert space $\mathcal{H}$ then 
\[
%\mathcal{H}_U:= \{ f\in \mathcal{H}: \|T^n f\| = \|T^{*n}\| = \|f\|_{\mathcal{H}}\, \forall n\in \mathbb{N}\},
\mathcal{H}_U := \displaystyle\cap_{k\geq 0}(\rm{fix}(\mathcal{T}^k\mathcal{T}^{*k})\cap\rm{fix}(\mathcal{T}^{*k}\mathcal{T}^k))
\]
is the largest space among all closed $\mathcal{T}$-invariant and $\mathcal{T}^*$-invariant subspaces of H on which $\mathcal{T}$ restricts to a unitary operator. The orthogonal complement $\mathcal{H}_U^\perp =\mathcal{H}_{NU} $ is the completely non unitary part of $\mathcal{H}$. Here \rm{fix}(A) refers to the subspace spanned by all the invariant vectors of operator $\rm{A}$.
\end{theorem}

Theorem \ref{thm:Wold} implies an orthogonal decomposition $f = f_{NU} + f_U$. Note that $\mathcal{H}_{f,NU}$ and $\mathcal{H}_{f,U}$ are invariant under $\mathcal{K}$ and $\mathcal{H}_{f,NU}\perp\mathcal{H}_{f,U}$. Then the fact that $\mathcal{H}_f$ is generated by $f$ implies that  $\mathcal{H}_{f,U} = \overline{\text{Span}_{\mathbb{C}}\{f_U,\mathcal{K}f_U,...\}}$.
Note that for any eigenvector $h\in\mathcal{H}_f$ of $\mathcal{K}$  associated to eigenvalue $\lambda$, we have an orthogonal decomposition $h = h_{f,U} + h_{f,{NU}}$, where $h_{f,U}\in\mathcal{H}_{f,U}$, $h_{f,{NU}}\in\mathcal{H}_{f,{NU}}$. Then $\mathcal{K}h = \lambda h = \mathcal{K}h_{f,U} +  \mathcal{K}h_{f,{NU}}$. Because $\mathcal{H}_{f,U}$ and $ \mathcal{H}_{f,{NU}}$ are invariant subspaces, $h =  \mathcal{K}h_{f,U}/\lambda +  \mathcal{K}h_{f,{NU}}/\lambda$ is an orthogonal decomposition for which $\mathcal{K}h_{f,U}/\lambda\in\mathcal{H}_{f,U}$, $\mathcal{K}h_{f,{NU}}/\lambda\in\mathcal{H}_{f,{NU}}$, implying that $h_{f,U}$ \YZ{ and } $ h_{f,NU}$ are both  \YZ{$\mathcal{K}-$}\YZ{eigenvectors} of the same eigenvalue. $h_{f,{NU}}$ must be zero because $\mathcal{K}$ acts on $\mathcal{H}_{f,{NU}}$ as the unilateral shift operator. Hence a \YZ{$\mathcal{K}-$eigenvector} 
\YZ{must be} inside $\mathcal{H}_{f,U}$.

\begin{definition}\label{def:f-cyc}
A Hilbert space $\mathcal{H}$ with an unitary operator $U$ is called $f$-cyclic if $\mathcal{H} = \overline{\text{Span}_{\mathbb{C}}\{f,Uf,...\}}$ for some $f\in\mathcal{H}$.
\end{definition}
\begin{theorem}[Spectral theorem for unitary operator]\label{thm:unitary spectral}
Let $\mathcal{H}$ be a Hilbert space and $U$ an unitary operator on $\mathcal{H}$. Assume that $\mathcal{H}$ is $f$-cyclic for some $f\in\mathcal{H}$. Then there exists a finite measure $\nu_f$ on the unit circle $S^1\subset\mathbb{C}$, and an isomorphism $\phi$
\begin{align}
 \phi: \mathcal{H} &\rightarrow L^2(S^1,d\nu_f)
\end{align}
such that  $\phi\circ U\circ\phi^{-1} (g)(z) = zg(z)$, for any $g\in L^2(S^1,d\nu_f)$ and any $z\in S^1$. In particular, $\phi(f) = 1$.
\end{theorem}
See lemma 5.4 in \cite{spectraltheoryBorthwick2020} for a mathematical proof. Note that lemma 5.4 in \cite{spectraltheoryBorthwick2020} assumes that $\mathcal{H} = \overline{\text{Span}_{\mathbb{C}}\{...,U^{-1}f,f,Uf,...\}}$, which is a weaker assumption than $\mathcal{H}$ being $f-$cyclic in the sense of Definition 1. Therefore Theorem \ref{thm:unitary spectral} applies to $\mathcal{H}_{f,U}$ and $\mathcal{K}$. In general, the spectrum measure $\nu$ consists of the discrete component, the singular-continuous component, and the absolutely continuous component (with respect to Lebesgue measure): $\nu = \nu_d + \nu_{sc} + \nu_{ac}$. The three components are pairwise-orthogonal, in the sense that for any $g\in L^2(S^1,d\nu)$, we can write $g = g_{d} + g_{sc} + g_{ac}$ such that $\nu_d(g_d) = \nu(g_d)$,  $\nu_d(g_{sc}) = \nu_d(g_{ac}) = 0$, similarly for $\nu_{sc}$ and $\nu_{ac}$.
Together with theorem \ref{thm:Wold}, this suggests the orthogonal decomposition  \YZ{$\mathcal{H}_f = \mathcal{H}_{f,d}+\mathcal{H}_{f,NU}+\mathcal{H}_{f,sc} + \mathcal{H}_{f,ac}$}, and 
\begin{align}
f = f_{NU} + f_{d} + f_{sc} + f_{ac}.\label{eq:f_decomp}
\end{align}
\YZ{It is easy to see that $\mathcal{H}_{f,d},\mathcal{H}_{f,NU},\mathcal{H}_{f,sc}$ and $\mathcal{H}_{f,ac}$ are invariant subspaces of $\mathcal{K}$.}
In particular, the discrete part $\nu_d$ is a finite or countable sum of Dirac measures $\nu_d= \displaystyle\sum_i |a_i|^2\delta_{\xi_i}$, where $\xi_i\in S^1$ is the support of $\delta_{\xi_i}$. Hence we can write
\begin{align}
    \phi(f_d) = \sum_{i} \mathds{1}_{\xi_i}, 
\end{align}
where $\mathds{1}_{\xi_i}(z) = 1$ if $z=\xi_i$ and 0 otherwise.

As such, for every $\xi_i\in\text{Supp}(\nu_d)$, $\phi^{-1}(\mathds{1}_{\{\xi_i\}})$ is an \YZ{eigenvector} of $\mathcal{K}$. On the other hand, let $h\in \mathcal{H}_{f,U}$ be an \YZ{eigenvector} of $\mathcal{K}$, i.e. $\mathcal{K}h = \xi h$ for some $\xi\in S^1$. Then $\|\xi\phi(h) - z\phi(h)\|^2 = 0$. Let $A = \{z:z\phi(h)(z)\neq \xi\phi(h)(z)\}$, then $\nu(|\mathds{1}_{A}\phi(h)|^2) = 0$, meaning that $\mathds{1}_A\phi(h) = 0$ in $L^2(S^1,d\nu)$. Hence $\phi(h) = \mathds{1}_{\{\xi\}}\phi(h)$, and $\nu(\{\xi\}) > 0$. This shows that there is a one-to-one correspondence between the support of the discrete measure $\nu_d$ and the \YZ{$\mathcal{K}-$eigenvectors} inside $\mathcal{H}_{f,U}$. \YZ{In particular, all the eigenvectors of $\mathcal{K}$ are simple.}

Let $v_i$ be the corresponding normalized  \YZ{$\mathcal{K}-$eigenvectors}, then we have 
\begin{align}
    f = \sum_{i}a_iv_i + f_{sc} + f_{ac} + f_{NU}.\label{eq:f decompose}
\end{align}

\YZ{
Our goal is to evaluate $|a_i|^2$, the numerical tool is the Gram matrix $G_{NM}(f)$, where the $(i,j)-$entry is 
\begin{align}
    G_{NM,ij}(f) = \frac{1}{M}\sum_{t=0}^Mf(i+t)\bar{f}(j+t).
\end{align}
We also define the autocovariance matrix
\begin{align}
    C_{N}(f)  :=& \begin{pmatrix}
    \rho_0, & \rho_1, & \cdots & \rho_{N} \\
    \bar{\rho}_{1}, & \rho_0, & \cdots & \rho_{N-1}\\
    \vdots \\
    \bar{\rho}_{N}, & \bar{\rho}_{N-1}, & \cdots & \rho_0
    \end{pmatrix}\nonumber \\
    =& \begin{pmatrix}
    \inp{f}{f}, & \inp{f}{\mathcal{K}f}, & \cdots & \inp{f}{\mathcal{K}^Nf} \\
    \inp{\mathcal{K}f}{f}, & \inp{\mathcal{K}f}{\mathcal{K}f}, & \cdots & \inp{\mathcal{K}f}{\mathcal{K}^{N}f}\\
    \vdots \\
    \inp{\mathcal{K}^Nf}{f}, & \inp{\mathcal{K}^Nf}{\mathcal{K}f}, & \cdots & \inp{\mathcal{K}^Nf}{\mathcal{K}^{N}f}
    \end{pmatrix}\nonumber
\end{align}
It is obvious that for any $N\in\mathbb{N}$ 
\begin{align}
    \lim_{M\to\infty}G_{NM}(f) = C_{N}(f).
\end{align}
Recall that $f = f_d + f_{NU} + f_{sc} + f_{ac}$ and that each of these four components is orthogonal to all other components. Hence 
\begin{align}
    C_{N}(f) = C_{N}(f_d) + C_{N}(f_{NU}) + C_{N}(f_{sc}) + C_{N}(f_{ac}).
\end{align}
}

\YZ{
Let $d_{NM,1} \geq d_{NM,2} \geq \dots \geq 0$ be the eigenvalues of $G_{NM}(f)$. Let $d_{N,1}\geq d_{N,2} \geq ...\geq 0$ be the eigenvalues of $C_{N}(f)$. It is clear that 
\begin{align}
    \lim_{M\to\infty}d_{NM,i} = d_{N,i}    
\end{align}
Our main result states that:
}
\begin{theorem}[Main result]\label{thm: main result gram}
Assume that \YZ{the autocovariance $\rho_s$ exists for all $s\in\mathbb{N}$}. Let $\{v_i\}$ be the time shift operator \YZ{eigenvectors} of unit length. Let $f \in \mathcal{H}_f$ the Hilbert space of observable time-series with $f = \displaystyle\sum_{i=1}^{\infty}a_iv_i + f_{sc} + f_{ac} + f_{NU}$, where $f_{sc}$ and $f_{ac}$ are the components of $f$ in the space spanned by the singular-continuous spectrum and absolute-continuous spectrum, and $f_{NU}$ the component of $f$ in the completely non unitary subspace (i.e. direct sum of unilateral shift spaces). Assume that $|a_1| \geq |a_2| \geq \dots \geq 0$. Then for any $i$:
\begin{align}
    \lim_{N\rightarrow\infty}\lim_{M\rightarrow\infty}\YZ{\frac{d_{NM,i}}{N}} = |a_i|^2.\label{eq:dnm_a}
\end{align}
\end{theorem}
\YZ{
For a given observable $f$, the trajectory matrix is defined as:
\begin{align}
    A_{NM}(f) 
    &= \begin{pmatrix}
    f(0), & f(1), & \dots & f(M) \\
    f(1), & f(2), & \dots & f(M+1) \\
    \vdots \\
    f(N), & f(N+1), & \dots & f(N+M)
    \end{pmatrix}.\label{A_NM}
\end{align}
Then it is clear that 
\begin{align}
    G_{NM}(f) = \frac{1}{M}A_{NM}(f)A_{NM}(f)^*
\end{align}
where $A^*$ refers to the conjugate transpose of $A$. Let $\delta_{NM,1}\geq \delta_{NM,2}\geq \dots \geq 0$ be the singular values of $A_{NM}(f)$. Then directly we have that $\delta_{NM,i} = \sqrt{Md_{NM,i}}$.
\begin{corollary}[Trajectory matrix version]\label{cor:main_traj}
Assume $\rho_s$ exists for all $s\in\mathbb{N}$. Then 
\begin{align}
    \lim_{N\to\infty}\lim_{M\to\infty}\frac{\delta_{NM,i}^2}{NM} = |a_i|^2.\label{eq:cnm_a}
\end{align}
\end{corollary}
}
\begin{remark}
The Gramian matrix is used by singular spectrum analysis methods (\cite{Ghil2002AdvancedSM}) to construct temporal modes of the given time series. The eigenfunctions of $G_{NM}$ are called temporal empirical orthogonal functions (EoFs). Theorem \ref{thm: main result gram}  implies that the leading temporal EoFs are \YZ{due to} theoretical \YZ{eigenfrequencies} \YZ{of $\mathcal{K}$}. Similarly, the data-adaptive harmonic decomposition (DAHD, \cite{Kondrashov2020DataadaptiveHA}) and Hankel alternative view of Koopman analysis (HAVOK, \cite{Brunton2017ChaosAA}) are based on Gramian  matrix and \YZ{trajectory} matrix, respectively. The main theorem and the corollary directly provides a way to identify which features extracted by SSA, DAHD, or HAVOK are related to  \YZ{eigenfrequencies of $\mathcal{K}$}, and which features are not.
\end{remark}

\YZ{The main result can be summarized as the following abstract mathematical theorem with respect to  isometric operators Hilbert space. According to our knowledge, this mathematical result is not shown in any previous literature.
\begin{theorem}[Main result in pure mathematical form]
Let $\mathcal{H}$ be a Hilbert space and $\mathcal{K}$ an isometry of $\mathcal{H}$. The inner product in $\mathcal{H}$ is denoted by $\inp{ }{ }$.  Let $f\in\mathcal{H}$ be any vector. Let 
\begin{align}
    C_{N}(f) = \begin{pmatrix}
    \inp{f}{f} & \inp{f}{\mathcal{K}f} & \cdots & \inp{f}{\mathcal{K}^Nf} \\
    \inp{\mathcal{K}f}{f} & \inp{\mathcal{K}f}{\mathcal{K}f} & \cdots & \inp{\mathcal{K}f}{\mathcal{K}^Nf}\\
    \vdots\\
    \inp{\mathcal{K}^Nf}{f} & \inp{\mathcal{K}^Nf}{\mathcal{K}f} & \cdots & \inp{\mathcal{K}^Nf}{\mathcal{K}^Nf}
    \end{pmatrix}.
\end{align}
Let $v_i$ be the eigenvectors of $\mathcal{K}$ in $\mathcal{H}$. And let
\begin{align}
    f = \sum_{i}a_iv_i + f^{\perp}
\end{align}
be the decomposition of $f$, where $f^{\perp}$ is perpendicular to the subspace spanned by all the eigenvectors of $\mathcal{K}$. There maybe uncountably many $v_i$ in $\mathcal{H}$ but only countably many are included in the summation.
Let $d_{N,1}\geq d_{N,2}\geq \cdots \geq 0$ be  the eigenvalues of $C_{N}$. And assume that $|a_1|\geq |a_2|\geq \cdots \geq 0$.
Then
\begin{align}
    \lim_{N\to\infty}\frac{d_{N,i}}{N} = |a_i|^2
\end{align}
\end{theorem}
}

\section{Proof of the main theorem}
We first present several lemmas which are independent of the language of Koopman theory.
\begin{lemma}\label{lem:contr}
Let \YZ{$\mathcal{T}$} be a contraction on a Hilbert space $H$. Then  %restriction of $T$ to $\mathcal{H}_{\EM{NU}}$ \EM{is weakly stable} as we have 
for every $f, g \in \mathcal{H}_{NU}$
\[
\lim_{n\to \infty}\langle \YZ{\mathcal{T}}^n f, g\rangle_{\mathcal{H}} =0
\]
\end{lemma}
\begin{proof}[Proof of lemma \ref{lem:contr}]
For every $f\in \mathcal{H}$ the sequence $(\|\YZ{\mathcal{T}}^n f\|)_{n\in \mathbb{N}}$ is decreasing thus convergent. For any $k\in\mathbb{N}$, we have
\begin{align*}
    &\|\YZ{\mathcal{T}}^{*k} \YZ{\mathcal{T}}^k \YZ{\mathcal{T}}^n f - \YZ{\mathcal{T}}^n f \|^2 = \|\YZ{\mathcal{T}}^{*k} \YZ{\mathcal{T}}^k \YZ{\mathcal{T}}^n f\| - \nonumber \\
    &2 {\mathcal Re}\left( \YZ{\mathcal{T}}^{*k} \YZ{\mathcal{T}}^k \YZ{\mathcal{T}}^n f, \YZ{\mathcal{T}}^n f\right) + \|\YZ{\mathcal{T}}^n f \|^2\\
    =& \|\YZ{\mathcal{T}}^{*k}\YZ{\mathcal{T}}^{n+k} f\|^2 -2 \|\YZ{\mathcal{T}}^{n+k} f\|^2 + \|\YZ{\mathcal{T}}^n f \|^2\\
    \leq& \|\YZ{\mathcal{T}}^{n+k} f\|^2 -2 \|\YZ{\mathcal{T}}^{n+k} f\|^2 + \|\YZ{\mathcal{T}}^n f \|^2 \\
    =& \|\YZ{\mathcal{T}}^n f \|^2 - \|\YZ{\mathcal{T}}^{n+k} f\|^2 \to 0 \text{ as } n\to \infty
\end{align*}
Hence  $\langle (I - \YZ{\mathcal{T}}^{*k}\YZ{\mathcal{T}}^k)\YZ{\mathcal{T}}^n f, g \rangle_{\mathcal H} \to 0$ for every $f,g \in {\mathcal H}$ as $n\to \infty$,
therefore,
\[
\langle \YZ{\mathcal{T}}^n f, g\rangle_{\mathcal H} \to 0  \text{ for every } g\in \rm{ran} (I - \YZ{T_c}^{*k}\YZ{T_c}^k) \text{ as } n\to \infty 
\]
The same argument for $\YZ{\mathcal{T}}^{*}$ yields
\[
\langle \YZ{\mathcal{T}}^n f, g\rangle_{\mathcal H}= \langle  f, \YZ{\mathcal{T}}^{*n} g\rangle_{\mathcal H} \to 0  \text{ for every } g\in \rm{ran} (I - \YZ{\mathcal{T}}^k\YZ{\mathcal{T}}^{*k})\]  as  $n\to \infty$.

We obtain that $\langle \YZ{\mathcal{T}}^n f, g\rangle_{\mathcal H}\to 0$ as $n\to \infty$ for every
\begin{align*}
    f,g &\in \sum_{k\geq 1}(\rm{ran} (I - \YZ{\mathcal{T}}^{*k}\YZ{\mathcal{T}}^k) + \rm{ran} (I - \YZ{\mathcal{T}}^k\YZ{\mathcal{T}}^{*k}))\\
    &= \left(\cap_{k\geq 1} \left(\rm{ran} (I - \YZ{\mathcal{T}}^{*k}\YZ{\mathcal{T}}^k)^{\perp} \cap \rm{ran} (I - \YZ{\mathcal{T}}^k\YZ{\mathcal{T}}^{*k})^{\perp}\right)\right)^{\perp}\\
    &= \left(\cap_{k\geq 1} (\rm{fix} (\YZ{\mathcal{T}}^{*k}\YZ{\mathcal{T}}^k) \cap \rm{fix} (\YZ{\mathcal{T}}^k \YZ{\mathcal{T}}^{*k})) \right)^{\perp} ={\mathcal H}_U^{\perp} = {\mathcal H}_{NU}
\end{align*}
\end{proof}

%\begin{corollary}\label{cor:shift_product}
%Let $\mathcal{H}_0$ be a Hilbert space over $\mathbb{C}$ of %finite dimension or infinite dimension. %$a_0,a_1,...\in\mathcal{H}_0$ such that %$\displaystyle\sum\|a_i\|^2 < \infty$. Then
%\begin{align}
%    \lim_{k\rightarrow\infty}\sum_{i}|\inp{a_{i}}{a_{i+k}}| = %0
%\end{align}
%\end{corollary}
%\begin{proof}[Proof of corollary \ref{cor:shift_product}]

%For any $\epsilon > 0$, there exists  $N$ such that $\displaystyle\sum_{i\geq N}\|a_i\|^2 \leq \epsilon/2$. Further there exists $M>N$, such that for any $i<N$ and $j>M$, $\|a_j\| < \frac{\epsilon}{2}\|a_i\|$. Now for any $k>M$,
%\begin{align}
%    &\sum_{i}|\inp{a_i}{a_{i+k}}| = \sum_{i=1}^N|\inp{a_i}{a_{i+k}}| + \sum_{i>N}|\inp{a_i}{a_{i+k}}|  \\
%    \leq& \sum_{i\leq N}\|a_i\|\|a_{i+k}\| + \frac{1}{2}\sum_{i>N}(\|a_i\|^2 + \|a_{i+k}\|^2)\nonumber \\
%    \leq& \frac{\epsilon}{2}\sum_{i\leq N}\|a_i\|^2 + \sum_{i>N}\|a_i\|^2\nonumber \\
%    \leq& \frac{\epsilon}{2} + \frac{\epsilon}{2} = \epsilon\nonumber
%\end{align}
%\end{proof}

\begin{lemma}\label{lem:shift_product}
Let $\mathcal{H}$ be a Hilbert space over $\mathbb{C}$ of finite dimension or infinite dimension. $a_0,a_1,...\in\mathcal{H}$ such that $\displaystyle\sum\|a_i\|^2 < \infty$. Then
\begin{align}
    \lim_{k\rightarrow\infty}\sum_{i}|\inp{a_{i}}{a_{i+k}}| = 0
\end{align}
\end{lemma}
\begin{proof}[Proof of Lemma \ref{lem:shift_product}]
Without loss of generality, we may assume that $\displaystyle\sum_{i}\|a_i\|^2 = 1$. For any $\epsilon > 0$, there exists  $N$ such that $\displaystyle\sum_{i\geq N}\|a_i\|^2 \leq \epsilon/2$. Further there exists $M>N$, such that for any $i<N$ and $j>M$, $\|a_j\| < \frac{\epsilon}{2}\|a_i\|$. Now for any $k>M$,
\begin{align}
    &\sum_{i}|\inp{a_i}{a_{i+k}}| = \sum_{i=1}^N|\inp{a_i}{a_{i+k}}| + \sum_{i>N}|\inp{a_i}{a_{i+k}}|  \\
    \leq& \sum_{i\leq N}\|a_i\|\|a_{i+k}\| + \frac{1}{2}\sum_{i>N}(\|a_i\|^2 + \|a_{i+k}\|^2)\nonumber \\
    \leq& \frac{\epsilon}{2}\sum_{i\leq N}\|a_i\|^2 + \sum_{i>N}\|a_i\|^2\nonumber \\
    \leq& \frac{\epsilon}{2} + \frac{\epsilon}{2} = \epsilon\nonumber
\end{align}
\end{proof}

\begin{lemma}[The weakly mixing property]\label{lem:weakly mixing}
Let $\nu$ be a continuous measure on $S^1$. Then for any $f,g\in L^2(S^1,d\nu)$ and $z \in S^1$
\begin{align}
    \lim_{N\to\infty}\frac{1}{N}\sum_{i=0}^N|\nu(z^if\bar{g})|^2 = 0
\end{align}
\end{lemma}
\begin{proof}
The proof of this mixing theorem can be found page 39 of \cite{Halmos1956}. An alternative proof is given in the appendix.
\end{proof}

\begin{lemma}\label{lem:nu_d1}
Let $L>0$ and $\xi_1,...,\xi_{L} \in S^1$ such that $\xi_i\neq \xi_j$, and $c_1\geq...\geq c_L > 0$. Then 
\begin{align}
    \lim_{N\rightarrow\infty}\max_{\sum_{i=1}^N|\alpha_i|^2 = 1}\frac{1}{N}\sum_{k=1}^Lc_k^2|\sum_{i=0}^N\alpha_i\xi_k^i|^2 = c_1^2 \label{eq:nu_d equiv}
\end{align}
\end{lemma}

\begin{proof}[Proof of lemma \ref{lem:nu_d1}]
Let 
\begin{align}
    E_{N} &= \frac{1}{\sqrt{N+1}}\begin{pmatrix}
    1 & 1 & \dots & 1 \\
    \xi_1 & \xi_2 & \dots & \xi_L \\
    \vdots \\
    \xi_1^N & \xi_2^N & \dots & \xi_L^N 
    \end{pmatrix}\nonumber \\
    &= (\Phi_{N,1},...,\Phi_{N,L}),\\
    \text{and }  C &= \text{diag}(c_1,...,c_L)
\end{align}
Then $\|\Phi_{N,k}\| = 1$ and $\inp{\Phi_{N,k}}{\Phi_{N,l}} =  \frac{1 - (\xi_k\xi_l^{-1})^{N+1}}{(N+1)(1-\xi_k\xi_l^{-1})}\rightarrow 0$ as $N\rightarrow \infty$.
Applying the Gram-Schmidt process to $E_N$, we get a matrix $\tilde{E}_N = E_NS_N$, so that the columns of $\tilde{E}_N$ are orthogonal to each other and of unit length. $S_N\rightarrow I_L$ as $N\rightarrow\infty$, where $I_L$ refers to the identity matrix. Hence $\displaystyle\lim_{N\rightarrow\infty}S_NC = C$.  As a consequence the leading singular value $d_{N,1}$ of $E_NC$ converges to $c_1$. On the other hand, for any $\alpha = (\alpha_0,...,\alpha_N)\in\mathbb{C}^N$ with $\|\alpha\|=1$, direct computation shows that 
\begin{align}
    \|\alpha^{\top} E_NC\|^2 = \frac{1}{N+1}\sum_{k=1}^Lc_k^2|\sum_{i=0}^N\alpha_i\xi_k^i|^2.
\end{align}
Hence 
\begin{align}
    \max_{|\alpha|^2 = 1}\frac{1}{N+1}\sum_{k=1}^Lc_k^2|\sum_{i=0}^N\alpha_i\xi_k^i|^2 = \|E_NC\|^2 = d_{N,1}^2\rightarrow c_1^2 
\end{align}
\end{proof}

Now we can start to prove the main result. Recall that $f = f_d + f_{sc} + f_{ac} + f_{NU}$, and \YZ{$C_{N}(f) = C_{N}(f_{NU}) + C_{N}(f_d) + C_{N}(f_{ac}) + C_{N}(f_{sc})$}.  For any \YZ{semi positive-definite} matrix \YZ{$A\in\mathbb{C}^{N\times N}$}, we define \YZ{$\|A\| = \displaystyle\max_{v\in\mathbb{C}^{N}}\frac{v^{\top}A\bar{v}}{\|v\|^2}$}. The maximal  \YZ{eigenvalue} $d_1$ of $A$ is equal to $\|A\|$. If Theorem \ref{thm: main result gram} holds for $i=1$, we can then recursively deduce Theorem \ref{thm: main result gram} for all $i$ by removing $a_iv_i$ from $f$ at each step. Since 
\begin{align}
    \lim_{M\to\infty}G_{NM}(f) = C_{N}(f),
\end{align}
It is thus sufficient to prove that:
\YZ{
\begin{align}
    &\lim_{N\rightarrow\infty}\frac{\|C_{N}(f_{NU})\|}{N} = \lim_{N\rightarrow\infty}\frac{\|C_{N}(f_{ac})\|}{N} \nonumber \\
    =& \lim_{N\rightarrow\infty}\frac{\|C_{N}(f_{sc})\|}{N} = 0,\label{eq:ANM/NM s sc ac}
\end{align}
}
and that
\begin{align}
    \YZ{\lim_{N\rightarrow\infty}\frac{\|C_{N}(f_{d})\|}{N}} = |a_1|^2.\label{eq:ANM/NM d}
\end{align}
Now fix $N$, for any \YZ{$g\in \mathcal{H}_f$},
\begin{align}
    \YZ{\|C_{N}(g)\|} &=\max_{\alpha}\frac{\|\sum_{i=0}^N\alpha_i\mathcal{K}^ig\|^2}{|\alpha_1|^2+...+|\alpha_N|^2}
\end{align}

Hence Eq.\eqref{eq:ANM/NM s sc ac} and \eqref{eq:ANM/NM d} are equivalent to the following:
\begin{align}
\lim_{N\rightarrow\infty}\max_{\alpha}\frac{\|\sum_{i=0}^N\alpha_i\mathcal{K}^ig\|^2}{N(|\alpha_1|^2+...+|\alpha_N|^2)} = \begin{cases}
|a_1|^2 \!\!\!\!\!\!&\text{if $g=f_d$ } \\
0 &\text{if $g = f_{sc}$,$f_{ac}$ or $f_{NU}$}.\label{eq:alpha X}
\end{cases}
\end{align}
The case when $g=f_{NU}$ can be quickly proved:

\begin{prop}[The case for $f_{NU}$]\label{prop:f_s case}
\begin{align}
\lim_{N\rightarrow\infty}\max_{\alpha}\frac{\|\sum_{i=0}^N\alpha_i\mathcal{K}^if_{NU}\|^2}{N(|\alpha_1|^2+...+|\alpha_N|^2)} = 0
\end{align}
\end{prop}
\begin{proof}[Proof of proposition \ref{prop:f_s case}]
Without loss of generality, we may assume that $\|\alpha\|=1$. Since $f_{NU}\in\mathcal{H}_{NU} = \bigoplus_{s\geq 0}\mathcal{H}_{s}$, we can write $f_{NU} = (a_0,a_1,...)$, where $a_i\in\mathcal{H}_{s}$. For $k>0$, let $c_{k} = |\inp{\mathcal{K}^if_{NU}}{\mathcal{K}^{i+k}f_{NU}}| = |\inp{\mathcal{K}^{i+k}f_{NU}}{\mathcal{K}^{i}f_{NU}}|$, which does not depend on $i$. Lemma \ref{lem:contr} implies that $\displaystyle\lim_{|i-j|\rightarrow\infty} c_{|i-j|} = 0$. Therefore for any $\epsilon>0$, there exists $M_{\epsilon}$ such that $c_{|i-j|} \leq \epsilon/4$ for any $|i-j|>M_{\epsilon}$. Now for any $N>2M_{\epsilon}\|f_{NU}\|^2/\epsilon$, and any $|\alpha_1|^2 + ...+|\alpha_N|^2 = 1$, 
\begin{align}
    &\|\sum_{i=0}^N\alpha_i\mathcal{K}^if_{NU}\|^2 = \sum_{i,j}\alpha_i\bar{\alpha}_j\inp{\mathcal{K}^if_{NU}}{\mathcal{K}^jf_{NU}}\nonumber \\
    \leq&2\sum_{k=0}^N\sum_{i=0}^{N-k}\alpha_i\bar{\alpha}_{i+k}c_k
    \leq\sum_{k=0}^N\sum_{i=0}^{N-k}(|\alpha_i|^2 + |\alpha_{i+k}|^2)c_k\nonumber \\
    \leq&2\sum_{k=0}^Nc_k\leq 2\sum_{k=0}^{M_{\epsilon}}c_k + 2\sum_{k>M_{\epsilon}}c_k\nonumber \\
    \leq& M_\epsilon \|f_{NU}\|^2 + (N-M_\epsilon)\epsilon/2 
    \leq \frac{N\epsilon}{2} + \frac{N\epsilon}{2} = N\epsilon
\end{align}
\end{proof}

Recall the notations in Theorem \ref{thm:unitary spectral}, for any $g\in\mathcal{H}_{f,U}$,
\begin{align}
    &\|\sum_{i=0}^N\alpha_i\mathcal{K}^ig\|_{\mathcal{H}_f}^2 = \|\sum_{i=0}^N\alpha_iz^i\phi(g)\|_{L^2(S^1,d\nu_f)}^2\nonumber \\
    =& \int_{S^1}|\sum_{i=0}^N\alpha_iz^i|^2|\phi(g)(z)|^2d\nu_f(z)
\end{align}
This proves that
\begin{prop}\label{prop:ANM-S1}
For any $g \in \mathcal{H}_{f,U}$,
\begin{align}
     &\YZ{\lim_{N\rightarrow\infty}\frac{\|C_{N}(g)\|}{N}}\nonumber \\
     =&\lim_{N\rightarrow\infty}\max_{\|\alpha\|=1}\frac{1}{N}\int_{S^1}|\sum_{i=0}^N\alpha_iz^i|^2|\phi(g)(z)|^2d\nu_f(z)
\end{align}

\end{prop}
To prove Eq.\eqref{eq:alpha X} for $g=f_d, f_{sc}$ and $f_{ac}$, we start with the following lemma.

\begin{lemma}\label{lem:nu_d2}
Let $f,\mathcal{H}_{f,U}, \nu_f, \phi$ be the same as in Theorem \ref{thm:unitary spectral}. For simplicity, we denote $\nu_f$ by $\nu$. $\nu = \nu_d+\nu_{sc} + \nu_{ac}$. Let $\nu_{d,1}$ be a purely discrete finite measure on $S^1$, such that $\{\xi_1,...,\xi_L\}=\text{Supp}(\nu_{d,1})\subset\text{Supp}(\nu_d)$ and $\nu_{d,1}(\{\xi_i\}) = \nu_{d}(\{\xi_i\})$ for any $0\leq i\leq L$. Let $c_k = \sqrt{\nu_d(\{\xi_k\})}$. Let $f_k = \phi^{-1}(\mathds{1}_{\{\xi_k\}})$, and set $h = \displaystyle\sum_{k=1}^Lf_k$. Let $d_{N,1}(h)$ be the leading \YZ{eigenvalue} of \YZ{$C_{N}(h)$}.
Then 
\begin{align}
    \YZ{\lim_{N\rightarrow\infty}\frac{d_{N,1}(h)}{N}} = \max_{k}c_k^2. \label{eq:nu_d normal}
\end{align}
\end{lemma}

\begin{proof}[Proof of proposition \ref{lem:nu_d2}]
According to proposition \ref{prop:ANM-S1},
\begin{align}
    &\YZ{\lim_{N\rightarrow\infty}\frac{d_{N,1}(h)}{N} = \lim_{N\rightarrow\infty}\frac{\|C_{N}(h)\|}{N}}\nonumber\\
    =&\lim_{N\rightarrow\infty}\max_{\|\alpha\|=1}\frac{1}{N}\int_{S^1}|\sum_{i=0}^N\alpha_iz^i|^2|\phi(h)(z)|^2d\nu_f(z)\nonumber \\
    =&\lim_{N\rightarrow\infty}\max_{\|\alpha\|=1}\frac{1}{N}\int_{S^1}|\sum_{i=0}^N\alpha_iz^i|^2d\nu_{d,1}(z)\nonumber \\
    =&\lim_{N\rightarrow\infty}\max_{\|\alpha\|=1}\frac{1}{N}\sum_{k=1}^Lc_k^2|\sum_{i=0}^N\alpha_i\xi_{k}^i|^2
\end{align}

Then lemma \ref{lem:nu_d1} implies what we want to prove.
\end{proof}

\begin{prop}[The case for $f_d$ ]\label{prop:f_d}
Eq.\eqref{eq:ANM/NM d} holds.
\end{prop}

\begin{proof}[Proof of proposition \ref{prop:f_d}]

For any $\epsilon>0$, we choose a truncation $\nu_d = \nu_{d,1} + \nu_{d,\epsilon}$, so that $\nu_{d,\epsilon}(S^1) < \epsilon$, $|\text{Supp}(\nu_{d,1})| < \infty$, and that $\nu_{d,1}(\{\xi_k\}) = \nu_d(\{\xi_k\})$ whenever $\xi_{k}\in\text{Supp}(\nu_{d,1})$. \YZ{This decomposition of measure also induces an orthogonal decomposition of $L^2(S^1,d\nu_d) = L^2(S^1,d\nu_{d,1}) \bigoplus L^2(S^1,d\nu_{d,\epsilon})$. Note that these two components are invariant under $\mathcal{K}$.} When $\epsilon$ is small enough, $|a_1|^2 = \max_{\xi}\nu_{d,1}(\xi)$.

Apply lemma \ref{lem:nu_d2} to $\nu_{d,1}$, and let $h$ be defined as in lemma \ref{lem:nu_d2}. Then $f = h+f_{\epsilon}$ and
\YZ{
\begin{align}
    C_{N}(f) = C_{N}(h) + C_{N}(f_\epsilon),
\end{align}
}
and
\begin{align}
    \YZ{\lim_{N\rightarrow\infty}\frac{\|C_{N}(h)\|}{N}} = c_1^2 = \max_{\xi}\nu_d(\{\xi\}) = |a_1|^2\label{eq:f_d 1}
\end{align}

And note that, applying Cauchy-Schwartz inequality, 

\begin{align}
    &\YZ{\lim_{N\rightarrow\infty}\frac{\|C_{N}(f_\epsilon)\|}{N}} \nonumber \\
    =& \lim_{N\rightarrow\infty}\max_{|\alpha|^2=1}\int_{S^1}\frac{|\sum_{i=0}^N\alpha_iz^i|^2}{N}d\nu_{d,\epsilon} \leq \nu_{d,\epsilon}(S^1) < \epsilon\label{eq:f_d 2}
\end{align}

Eq.\eqref{eq:f_d 1}\eqref{eq:f_d 2} implies Eq.\eqref{eq:ANM/NM d} by letting $\epsilon\rightarrow 0$.
\end{proof}

\begin{prop}\label{prop:f_c}
Let $\nu_c$ be a continuous finite measure on $S^1$. Then 
\begin{align}
    \lim_{N\rightarrow\infty}\max_{\|\alpha\|^2=1}\int_{S^1}\frac{1}{N}|\sum_{i=0}^N\alpha_iz^i|^2d\nu_c = 0.
\end{align}
\end{prop}
\begin{proof}[Proof of proposition \ref{prop:f_c}]
Let $c_k = |\nu_c(z^k)|$.  In lemma \ref{lem:weakly mixing}, let $f=g=1$, it implies that
\begin{align}
    \lim_{N\to\infty}\frac{1}{N}\sum_{k=0}^N c_k^2 = 0.
\end{align}
Therefore for $\|\alpha\| = 1$,
\begin{align}
    &\int_{S^1}\frac{1}{N}|\sum_{i=0}^N\alpha_iz^i|^2d\nu_c 
    \leq\frac{2}{N}\sum_{k=0}^N\sum_{i=0}^{N-k}|\alpha_{i+k}\bar{\alpha}_{i}|c_k \nonumber \\
    \leq& \frac{1}{N}\sum_{k=0}^N\sum_{i=0}^{N-k}(|\alpha_{i+k}|^2 + |\alpha_{i}|^2)c_k\leq \frac{1}{N}\sum_{k=0}^N2c_k\nonumber \\
    \leq&\frac{2}{N}\sqrt{(N+1)\sum_{k=0}^Nc_k^2} = 2\sqrt{\frac{N+1}{N}\frac{1}{N}\sum_{k=0}^Nc_k^2}\to 0, 
\end{align}
as $ N\to\infty.$
\end{proof}

%\begin{proof}[Proof of proposition \ref{prop:f_c}]
%We prove by contradiction. We assume that proposition \ref{prop:f_c} does not hold. Then there exists $\epsilon > 0$ and a sequence of $\alpha^{(i)} = \{\alpha^{(i)}_{1},...,\alpha^{(i)}_{n_i}\}$, such that $\|\alpha^{(i)}\|^2 = 1$ and 
%\begin{align}
%    \nu_c(|P^{(i)}|^2/n_i) > \epsilon,
%\end{align}
%where $P^{(i)}(z) = \displaystyle\sum_{j=0}^{n_{i}}\alpha_{j}^{(i)}z^j$.
%Note that $|P^{(i)}(z)|^2/n_i \leq 1$ for any $z\in S^1$. Let 
%\begin{align}
%    A^{(i)} = \{z\in S^1 \Big{|} |P^{(i)}(z)|^2 >= \frac{\epsilon}{2}n_i\}.
%\end{align}
%Then $\nu_s(A^{(i)}) \geq \frac{\epsilon}{2}$.

%Now we pick any $L>4/\epsilon$. Lemma \ref{lem:seq} implies that we can find $\xi_1,...,\xi_L\in S^1$, $\xi_{i}\neq\xi_j$ for $i\neq j$, and a subsequence $A^{(i_j)}$ of $A^{(i)}$, such that $\xi_{k}\in A^{(i_j)}$ for any $k,j$. Then for any $i_j$, set $N = n_{i_j}$ and $\alpha = \alpha^{(i_j)}$,
%\begin{align}
%    \sum_{k=1}^L\frac{1}{N}|\sum_{i=0}^N\alpha_i\xi_{k}^i|^2 \geq \sum_{k=1}^L\frac{1}{N}\frac{\epsilon}{2}N \geq 2
%\end{align}
%This contradicts with lemma \ref{lem:nu_d1} by setting $c_1=c_2=...=c_L = 1$.
%\end{proof}

\begin{corollary}[The case for $f_{sc}$ and $f_{ac}$]\label{cor:f_c}
Eq.\eqref{eq:ANM/NM s sc ac} holds for $f_{sc}$ and $f_{ac}$.
\end{corollary}
\begin{proof}
This is the direct consequence of proposition \ref{prop:ANM-S1} and \ref{prop:f_c}.
\end{proof}

\section{Algorithm and Discussion}
A direct application of the main theorem is to determine whether or not the given finite data set is sufficient enough for \YZ{determining the $i-$th \YZ{$\mathcal{K}-$}eigenfreqeuncy} using \YZ{Gramian} matrix. For this purpose, we provide the following algorithm.
\begin{itemize}
    \item Given a time series of data $\{f(t)\}_{0\leq t\leq T}$ \YZ{where $t,T$ are non-negative integers that represent the iteration number}, choose $N_k, M_{k,j}$ where $1\leq j \leq l_k$, such that $N_k+M_{k,j}\leq T$, $M_{k,1}< M_{k,2}< ...< M_{k,l_k}\gg N_k$.
    \item For each $N_k, M_{k,j}$, compute the renormalized \YZ{eigenvalues} of $G_{N_kM_{k,j}}$, denoted by \YZ{$\sigma_{k,j,i} = \frac{d_{N_kM_{k,j},i}}{N_k}$}.
    \item Given $i$, for each $N_k$ check if $\sigma_{k,j,i}$ converges as $j$ increases. If for some $k$ it does not converge, it means that the $i-$th  \YZ{$\mathcal{K}$}\YZ{eigenfrequency} is not well represented by this dataset.
    \item Given $i$, if for all $k$, $\sigma_{k,j,i}$ shows good convergence, then check if $\sigma_{k,l_k,i}$ converges as $k$ increases. If $\sigma_{k,l_k,i}$ converges to some nonzero number, then the \YZ{energy of the} $i-$th  \YZ{$\mathcal{K}-$}\YZ{eigenfrequency} is well represented by this data set. Otherwise, the $i-$th  \YZ{$\mathcal{K}-$}\YZ{eigenfrequency} is not well-represented by this data set.  
\end{itemize}
\YZ{
\begin{remark}[Identification  of the  \YZ{$\mathcal{K}-$}eigenfrequencies]
Assume convergence for sufficiently enough $i$, choose $G_{NM}$ so that $d_{NM,i}/N$ is close enough to $|a_i|^2$ for sufficiently many $i$. Assume that $|a_{k-1}| > |a_k| = |a_{k+1}| = ... = |a_{k+L}|>|a_{k+L+1}|$. $L$ must be finite because $\|f\|^2 \geq \displaystyle\sum_{i}|a_i^2|$. Let $\xi_k,...,\xi_{k+L}$ be the corresponding theoretical $\mathcal{K}$eigenfrequencies. Our goal is to identify $\xi_{k},...,\xi_{k+L}$. Let $\eta_{k+i} = (1,\xi_{k+i},\xi_{k+i}^2,...,\xi_{k+i}^N)$. Let $\{v_{NM,k},v_{NM,k+1},...,v_{NM,k+L}\}$ be the corresponding eigenvectors of $G_{NM}$. Then each of $v_{NM,k},...,v_{NM,k+L}$ is approximately a linear combination of $\eta_{k},...,\eta_{k+L}$. Then $\xi_{k+i}$ can be identified by applying Fourier analysis. In following two cases, the eigenfrequency can be approximated by counting the local maximums of $v_{NM,i}$. 
\begin{itemize}
    \item Case 1: $f$ is a real valued observable. And for each $v_i$ there is no eigen-vector except the conjugate of $v_i$ that has the same energy as $v_i$;
    \item Case 2: for each $v_i$ there does not exist other eigen-vectors that has the same energy as $v_i$.
\end{itemize}
\end{remark}  
}

\subsection{Implication to Hankel DMD}
In \cite{Arbabi2017} a Hankel DMD algorithm has been proposed and the authors showed that $\YZ{d_{N,i}}$ can be used to identify Koopman and non-Koopman eigenfunctions for fixed $N$  under the conditions that 1), the Hilbert space $\mathcal{H}_f$ is finite dimensional and 2), $N$ is  larger than the dimension of $\mathcal{H}_f$. More precisely, they showed that $\YZ{d_{N,i}} > 0$ if and only if \YZ{$\YZ{d_{N,i}}$} corresponds to a Koopman eigenfunction. However, this assumption is already too strong even for the case where $f(x)$ is the observation of the first component of the 3 dimensional Lorenz system. In the case for which the dimension of $\mathcal{H}_f$ is infinite, their method unfortunately fails. Because $d_{N,i}$ can be positive even if there is no Koopman eigenfunctions. Therefore Theorem \ref{thm: main result gram} can be thought of as a completion of the method posed in \cite{Arbabi2017} under a much weaker assumption, by letting $N\to\infty$.
\subsection{Comparison with Yosida's formula}
Yosida mean ergodic theorem \cite{Yosida1995} provides a formula to calculate $a_\omega$, the coefficient of the Koopman eigenfunction of frequency $\omega$ in Eq.\eqref{eq:f decompose}:
\begin{align}
    a_{\omega} = \lim_{T\to\infty}\frac{1}{T}\sum_{t=0}^{T-1}\exp(-2\pi i\omega t)f(t). \label{eq:Yosida}
\end{align}
$a_{\omega} = 0$ if $\omega$ is not a Koopman eigenfrequency. Under the assumption of ergodicity and finite invariant measure, this formula can be proved by combining Theorem \ref{thm:Wold}, lemma \ref{lem:shift_product}, and Von-Neumann ergodic theorem. \YZ{In the more general situation where only the existence of autocovariances is assumed, we do not know whether the limit in Eq.\eqref{eq:Yosida} always exists. Nor do we know if the output of Yosida's formula is strictly related to Koopman theory. } This formula was first introduced to the fluid dynamics' community by \cite{Mezi2004ComparisonOS, Mezi2005SpectralPO}. Eq.\eqref{eq:Yosida} is easy to compute for \YZ{finite $T$ and} a given $\omega$. In the case for which the Koopman eigenfrequencies are unkown, numerically one still has the chance to identify some Koopman eigenfrequencies by calculating Eq.\eqref{eq:Yosida} for all $\omega\in\{k\Delta\omega: k = 1,...,n\}$ and then finding the peak value.

On the other hand, from the theoretical point of view, \YZ{when the system is ergodic and has a finite invariant measure,}  our result allows us to identify the Koopman eigenfunctions without having  prior knowledge about the Koopman eigenfrequencies. 

%\YZ{
%\subsection{Connection with singular spectrum analysis (SSA)}
%The community of SSA call the matrix $A_{NM}(f)$ the trajectory matrix. The Gramian matrix, which is also called lag-covariance matrix, is used to for singular spectrum analysis (see for instance \cite{Ghil2002AdvancedSM}). Here we briefly sketch how SSA is processed.
%\begin{itemize}
%    \item Given a time series $\{f(t): t= 0,1,2,..,T\}$,  choose proper $N,M$ and construct $G_{NM}(f)$ using Eq.\eqref{eq:def_gramian}.
%    \item Compute the eigendecomposition of $G_{NM} = USU^{*}$. The left eigen-vectors of $G_{NM}$ are the column vectors of $U$, denoted by $v_0,...,v_{N}\in\mathbb{R}^{N+1}$.
%    \item Choose a subset $K\subset\{0,1,2,...,N\}$.
%    \item Calculate $a^{\text{(ssa)}}_k(t) = \displaystyle\sum_{j=0}^Nf(t+j)\bar{v}_k(j)$ for $k\in K$, where $\bar{v}$ refers to the complex conjugate of $v$.
%    \item Choose $L_t$, $U_t$ and $M_t$ (see for instance Eq.(12) in \cite{Ghil2002AdvancedSM}).
%    \item The reconstructed components are given by 
%    \begin{align}
%        R_{K}(t) = \frac{1}{M_t}\displaystyle\sum_{k\in K}\sum_{j=L_t}^{U_t}a^{\text{(ssa)}}_k(t-j)v_{k}(j),
%    \end{align}
%\end{itemize}
%}

\section{Numerical experiments}

\subsection{Lorenz63 system}
To first test the theorem-based methodology, we consider the Lorenz63 system. We integrate Lorenz system using the Runge-Kutta 4th order scheme with $\Delta t = 0.01$. As already mentioned in \cite{Das2017}, \YZ{this system is ergodic and has finite invariant measure. Hence the autocovariance always exist \YZ{and $\mathcal{H}_f$ can be identified with a subspace of $L^2(X,d\mu)$ and $\mathcal{K}$ coincides with the classical Koopman operator on $\mathcal{H}_f$.}} Due to its weakly mixing nature, the only Koopman eigenfunction of Lorenz 63 is the constant function which has frequency 0. Let $f(t) = x - \bar{x}$, where $x$ is the first component of Lorenz system and $\bar{x}$ is the temporal mean of $x$. We use $E_{H,1}(N,M)$ to denote the leading renormalized singular value $\frac{d_{NM,1}^2}{NM}$. Then the decomposition $f = \displaystyle\sum_{i}a_{i}\YZ{v}_i + f_{NU} + f_{sc} + f_{ac}$ can be reduced to $f= f_{NU} + f_{sc} + f_{ac}$.  As expected, Fig.\ref{fig:L63} does not display the tendency that $E_{H}(N,\YZ{M_{\text{max}}})$ converge to some nonzero value as $N\to\infty$.
\begin{figure}
    \centering
    \includegraphics[width = 0.5\textwidth]{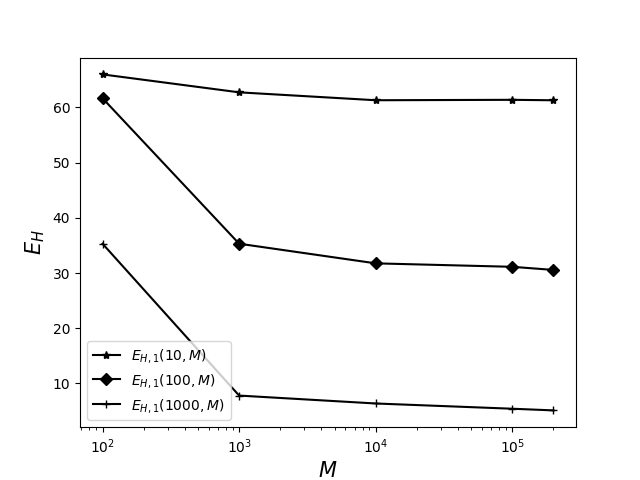}
    \caption{The results about Lorenz 63 system. The behavior of $E_{H,1}(N,M)$ agrees with the theoretical fact that $E_{H}(N,\YZ{M_\text{max}})$ does not converge to any nonzero value as $N\to\infty$. }
    \label{fig:L63}
\end{figure}

\subsection{A simple 4-dimensional system}
Following one of the numerical examples in \cite{Das2017}, we next consider a coupled system $(X,T, \mu)$, which consists of the discrete-time Lorenz system $(X_{63}, T_{63}, \mu_{63})$ and a rotation on the unit circle $(S^1,T_1, \mu_1)$, i.e. $X = X_{63}\times S^1$, $T = T_{63}\times T_1$ and $\mu = \mu_{63}\times\mu_1$. It is outlined in \cite{Das2017} that  $\mu$ is an invariant measure. Still, the Lorenz system does not have non-trivial Koopman \YZ{eigenfrequency} and $(X,T,\mu)$ is ergodic. \YZ{Therefore the autocovariance exists \YZ{and $\mathcal{K}$ coincides with the classical Koopman operator.}.} 

We choose the rotation $T_1$ to have period $p = \pi/5$ and define the observable
\begin{align}
    f(x,y,z,\xi)=\sin(\xi+x/10). 
\end{align}
For simplicity, we also use $f(t)$ to denote $f(x(t),y(t),z(t),\xi(t))$.
Then $f=\displaystyle\sum_{i}a_iv_i + f_{sc} + f_{ac} + f_{NU}$ as in Eq.\eqref{eq:f decompose}.

%\begin{align}
%    |a_\omega|^2 = \lim_{T\to\infty}|\frac{1}{T}\sum_{t=1}^T\exp(-2\pi i\omega t)f(t)|^2.\label{eq:Yosida^2}
%\end{align}

Anticipated by our main theorem, the renormalized singular value of the \YZ{trajectory} matrix should then converge to the same quantity as the one calculated by Eq.\eqref{eq:Yosida}. Hence it is worth to make a numerical comparison about $|a_{\omega}|^2$ obtained from Yosida's formula and that from the singular values of $A_{NM}$. 

The integration time step for Lorenz system is $\Delta t = 0.01$. The Runge-Kutta 4th order scheme is applied for integrating Lorenz system. The frequency $\omega$ we investigate is exactly the inverse of the period of $(S^1,T_1)$, i.e. $\omega = 5\Delta t/\pi$ for Eq.\eqref{eq:Yosida}. We use $E_{Y}(T)$ to denote the value of $|a_{\omega}|^2$ computed by Eq.\eqref{eq:Yosida}, and $E_{H,i}(N,M)$ to denote the value of $\frac{d_{NM,i}^2}{NM}$ which is computed from the singular value decomposition of $A_{NM}$. Note that $t,T,N,M$ are all integers which refer to the number of time steps instead of the exact time. 

\begin{figure}
    \centering
    \includegraphics[width = 0.5\textwidth]{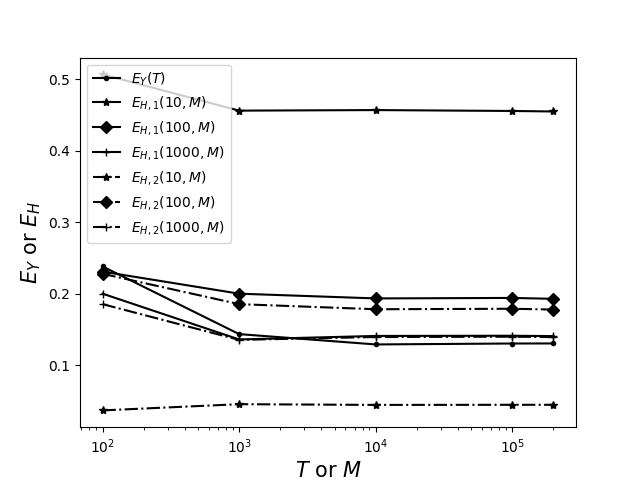}
    \caption{The numerical results of $E_Y(T)$ and $E_{H,i}(N,M)$ for $(X_{63}\times S^1, T_{63}\times T_1, \mu_{63}\times \mu_1)$.}
    \label{fig:L63+S1}
\end{figure}

Fig.\ref{fig:L63+S1} shows the numerical results of $E_{Y}(T)$ and $E_{H}(N,M)$. The value $E_Y(T)$ does not converge to $0$, showing that $\omega$ is indeed a Koopman eigen frequency. We also computed $\|f\|^2\approx\frac{1}{T}\displaystyle\sum_{t=1}^T|f(t)|^2 \approx 0.5002$. $E_{Y}(2\times 10^5)\approx 0.1303$, meaning that the fraction of energy in $f$ represented by the Koopman \YZ{eigenvector} $v_{\omega}$ is about $26\%$. Note that $E_H(10^3,2\times 10^5)$ is close to $E_{Y}(2\times 10^5)$, meaning that the leading singular value of the \YZ{trajectory} matrix indeed corresponds to the eigenfrequency $\omega$. $E_{H,1}(10^3,M)$ and $E_{H,2}(10^3,M)$ seem to converge to the same value. This is because the Koopman \YZ{eigenfrequencies} always exist in pair, i.e. $\exp(2\pi i\omega)\in\text{Supp}(\nu_d)\iff \exp(-2\pi i\omega)\in\text{Supp}(\nu_d)$. Since the observable $f$ is real, the coefficient $a_{\omega} = \bar{a}_{\bar{\omega}}$. Therefore, the total fraction of energy in $f$ that is represented by signals of period $p=\frac{\pi}{5}$ is about $52\%$.

\subsection{AVISO (DUACS) interpolated ocean topography data (1993-2019)}
For final illustration, we consider sea surface height (SSH) estimates. The AVISO gridded products provide the global SSH interpolation since 1993, the year after the launch of the first satellite altimeter TOPEX/Poseidon. The SSH is interpolated daily at a grid resolution of $0.25^\circ\times 0.25^{\circ}$. 
%This dataset, which has also been improved along with time as the interpolation method gets improved by researchers, has been used as one of the major datasets for the study of SSH. 
In this subsection, we use the main theorem to possibly assess the use of Koopman analysis for this dataset.

\YZ{The} assumption \YZ{of the main theorem} is the \YZ{existence of autocovariance}, which implies that the system should be stationary. We thus process the data by removing the overall constant rising tendency of SSH at each grid point over the decades (see for instance Fig.2 in \cite{Cazenave2010}). \YZ{We can not assert that the whole Earth system is ergodic and has an invariant measure, which includes the Earth, the atmosphere, ocean, all celestial bodies, but also the biology and living animals, etc. Hence we can not claim that the quantities $a_i$, $v_i$, etc. are associated to the classical Koopman operator. But, as we already stated, all these quantities are well-defined mathematically as long as the autocovariances exist. Moreover, in order to apply Yosida's formula, we assume that the eigenvectors $v_i\in\mathcal{H}_f$ can be represented as a time series of the form $\{1, e^{i\omega}, e^{2i\omega},...\}$ for some $\omega\in [0,2\pi]$. In this case, the output of Yosida's formula $a_{\omega}$ must be $a_i$.}  
We renormalize the SSH at each grid point, to simply ensure the data to have zero mean and unit variance at every grid point. We first apply Yosida's formula (Eq.\ref{eq:Yosida}) to the global data to compute $|a_{\omega}|$ at every grid point, where $\omega = \exp(2\pi i /365.25)$. This quantity is computed for January 1, 1998, i.e. $f(1)$ refers to the SSH at Jan. 1, 1998. Note that in theory, i.e. \YZ{assuming the autocovariance exists} and the data set large enough, this quantity does not depend on time. Since the data now has unit variance, $|a_{\omega}|^2$ can be interpreted as the fraction of energy in the SSH that is represented by the  \YZ{$\mathcal{K}-$eigenfrequency} $\omega$. Similarly, since the SSH are real numbers, $|a_\omega| = |a_{\YZ{-\omega}}|$ and $|a_{\omega}|^2 + |a_{\YZ{-\omega}}|^2 = 2|a_{\omega}|^2$ represents the fraction of energy represented by the yearly signal. 

\begin{figure}
    \centering
    \includegraphics[width=0.5\textwidth]{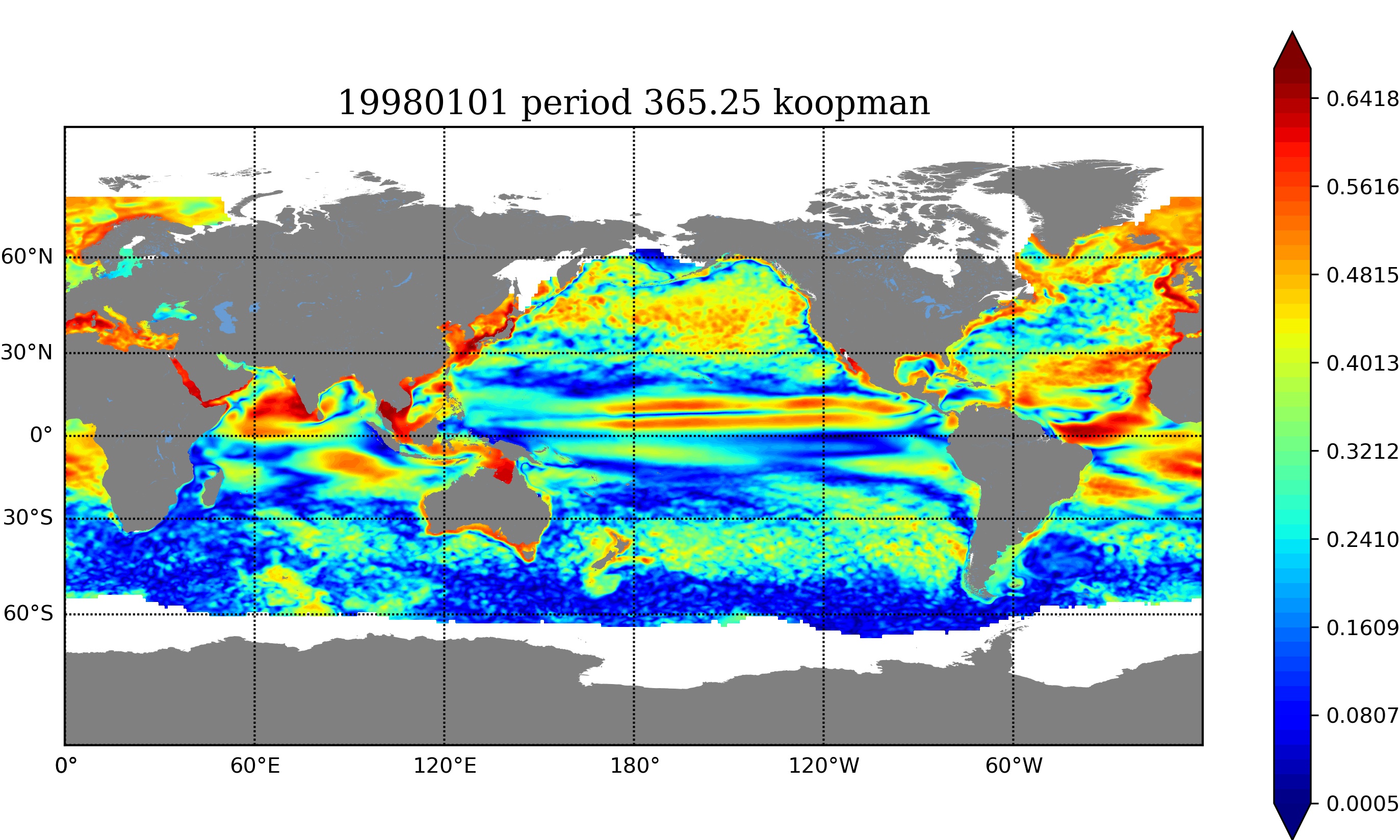}
    \caption{ The $|a_{\omega}|$ value at each grid point computed at Jan. 1, 1998 using Eq.\eqref{eq:Yosida} }
    \label{fig:global_yearly}
\end{figure}

Fig.\ref{fig:global_yearly} shows that more than $0.5^2+0.5^2 = 50\%$ of energy at the pacific ocean to the north of the equator (for instance at ($114.875^{\circ}$W, $6.125^{\circ}$N)) is represented by the yearly signal. Constructing the \YZ{trajectory} matrix for SSH at ($114.875^{\circ}$W, $6.125^{\circ}$N)), i.e. we choose $f = \text{SSH}(114.875^{\circ}$W,$6.125^{\circ}$N$)$, we can then compare $E_{Y}(T)$ and $E_{H,i}(N,M)$, for $i=0,1,2,3$, $N = p,3p,6p$, $M = 3p,6p,20p$, with $p = 1$ year $= 365.25$(days). Fig.\ref{fig:SSH 1,2} shows that the first two renormalized singular values apparently converge to the fraction of energy represented by \YZ{$\mathcal{K}-$eigenfrequencies} $\omega$ and $\YZ{-\omega}$. However, the third and fourth renormalized singular values do now show a sign of convergence. As shown in Fig.\ref{fig:SSH 3,4}, this is likely due to the overall limited length of the present-day data set regarding the high dimensional  \YZ{state space of the} dynamical \YZ{system} at stake.

\begin{figure}
    \centering
    \includegraphics[width = 0.5\textwidth]{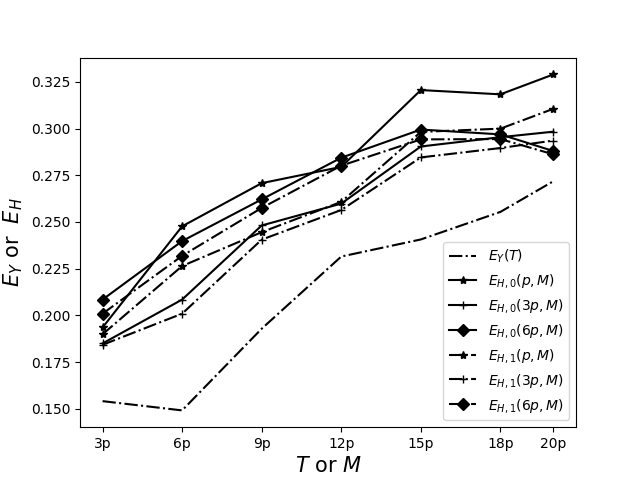}
    \caption{The numerical results of $E_Y(T)$ and $E_{H,i}(N,M)$ (i=0,1) for AVISO interpolated SSH at ($114.875^{\circ}$W, 6.125$^{\circ}$N). It shows that the first and the second singular value have possibly converged.}
    \label{fig:SSH 1,2}
\end{figure}

\begin{figure}
    \centering
    \includegraphics[width = 0.5\textwidth]{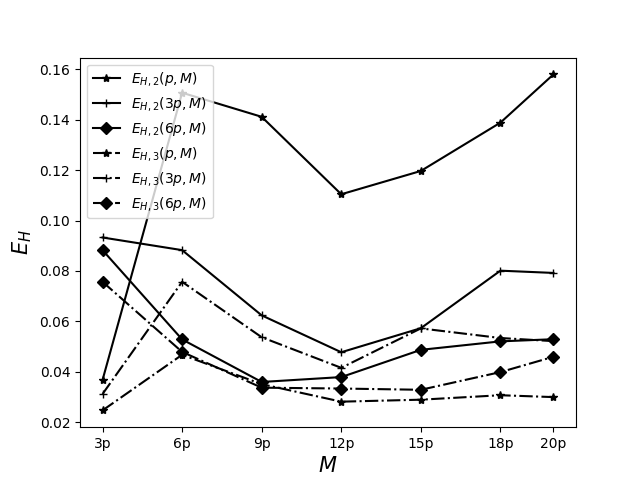}
    \caption{The numerical results of $E_Y(T)$ and $E_{H,i}(N,M)$ (i=2,3) for AVISO interpolated SSH at ($114.875^{\circ}$W, 6.125$^{\circ}$N). It shows that the third and the fourth singular value have not yet converged.}
    \label{fig:SSH 3,4}
\end{figure}

\section{Conclusion}
The main objective of this study is to provide a rigorous and practical method to identify Koopman \YZ{eigenfrequencies} for discrete-time ergodic and (finite) measure preserving  dynamical systems. \YZ{In the more general situation where we only assume that the time series has well-defined autocovariances, we define the Hilbert space $\mathcal{H}_f$ purely based on the time series and the time shift  operator $\mathcal{K}$ that acts on $\mathcal{H}_f$. When the system is ergodic and has finite invariant measure, $\mathcal{H}_f$ can be identified with the closure of the Krylov subspace generated by an observable $f$, and the time-shift operator $\mathcal{K}$ coincides with the classical Koopman operator on the observable space $\mathcal{H}_f$}. This work follows the result in \cite{Arbabi2017}, but further extend the applicability of the Hankel-DMD. It  provides a theorem-based practical way to help assess the results of the decomposition in terms of Koopman \YZ{eigenfrequencies}.  \YZ{It shows that} the leading temporal EoFs, which are calculated from the eigen decomposition of the Gramian matrix, are \YZ{indeed due to intrinsic eigenfrequencies}. The main theorem provides partial theoretical foundation to several existing empirical methods including SSA,DAHD, and HAVOK. % two assumptions hold for each specific system of interest. 
The main result shows that the discrete spectrum \YZ{of $\mathcal{K}$} can be characterized by the singular values (eigenvalues) of \YZ{trajectory} (Gramian, respectively) matrix. It remains to study whether the continuous spectrum can also be characterized by these matrices.

The numerical illustrations demonstrate the applicability of the theorem-based methodology for low dimensional systems. Yet, using sea surface height observables to inform about a very large dimension dynamical planet system, it is also apparent that one major difficulty of applying the main theorem might be the length of the data-set. An heuristic solution is to possibly associate the observables at different grid points, and/or to consider multiple observables, i.e. sea surface temperature. We reserve these investigations for future studies.

\section*{Acknowledgement}
The authors acknowledge the support of the ERC EU project 856408-STUOD, the support of the ANR Melody project, the support from China Scholarship Council, and the support from the National Natural Science Foundation of China (Grant No. 42030406).

\appendix
\section{An alternative proof of the weakly mixing property}
In this appendix we provide an alternative proof of the weakly mixing property. Note that the proof of mixing theorem on page 39 of \cite{Halmos1956} implies that the weakly mixing property is equivalent to 
\begin{align}
    \lim_{N\to\infty}\frac{1}{N}\sum_{i=0}^N|\int_{S^1}z^id\mu(z)|^2 = 0,
\end{align}
for any continuous measure $\mu$, which apparently is equivalent to
\begin{align}
    \lim_{n\rightarrow\infty}\frac{\#\{0<k<n : |\int_{S^1}z^k d\mu(z)|>\epsilon\}}{n} = 0,\label{eq:moments}
\end{align}
for any $\epsilon > 0$ and continuous measure $\mu$. We shall provide an alternative proof for proposition \ref{prop:f_c} and then derive Eq.\eqref{eq:moments} from proposition \ref{prop:f_c}. To do this, we need the following lemma.

\begin{lemma}\label{lem:seq}
Let $\nu_c$ be a continuous finite measure on $S^1$, and $\{A^{(i)}\}_{i=1}^{\infty}$ a sequence of subsets of $S^1$ such that $\nu_c(A^{(i)}) > \epsilon$ for some fixed $\epsilon > 0$ and for any $i$. Then for any $L>0$, there exists $\xi_{1},...,\xi_{L} \in S^1$, and a subsequence $\{A^{(i_k)}\}_{k=1}^{\infty}$, such that $\xi_{j}\in A^{(i_k)}$ for any $j$ and $k$.
\end{lemma}
\begin{proof}[Proof of lemma \ref{lem:seq}]
    The idea of the proof is that we first show that there exists a point $\xi_1\in S^1$ and $\Lambda_1\subset\mathbb{N}$ such that $|\Lambda_1| = \infty$ and $\xi_1\in A^{(i)}$ for any $i\in\Lambda_1$. Then we choose a small neighborhood $I_1$ of $\xi_1$ so that $\nu_c(I_1)<\epsilon/2$. This can be done merely because $\nu_c$ is a continuous measure. Let $A^{(i),1} = A^{(i)}-I_1$ for $i\in\Lambda_1$, we have $\nu_c(A^{(i),1}) >\epsilon/2$. Then we apply the same analysis to $\{A^{(i),1}\}_{i\in\Lambda_1}$ to find $\xi_2$, etc. After doing the same analysis for $L$ times, we get $\xi_1,...,\xi_L$ and $\Lambda_L$ such that $|\Lambda_L|=\infty$ and $\xi_j\in A^{(i)}$ for any $j$ and $i\in\Lambda_L$.
    
    To prove that there exists a point $\xi_1\in S^1$ and $\Lambda_1\subset\mathbb{N}$ such that $|\Lambda_1| = \infty$ and $\xi_1\in A^{(i)}$ for any $i\in\Lambda_1$. We prove by contradiction. Suppose that this is not true, i.e. for any $\xi\in S^1$ there exists $N$ such that $\xi\notin A^{(i)}$ for any $i>N$. Let $B_N = \cup_{i\geq N}A^{(i)}$. Then $B_1\supset B_2\supset\dots$ and $\cap_{N\geq 1}B_N = \emptyset$. It means that $\nu_c(B_N)\to 0$ as $N\to\infty$. This is apparently not true because $\nu_c(B_N) \geq \nu_c(A^{(N+1)}) > \epsilon$.
\end{proof}
Now we give another proof of proposition \ref{prop:f_c} based on lemma \ref{lem:seq}.
\begin{proof}[An alternative proof of proposition \ref{prop:f_c}]
We prove by contradiction. We assume that proposition \ref{prop:f_c} does not hold. Then there exists $\epsilon > 0$ and a sequence of $\alpha^{(i)} = \{\alpha^{(i)}_{1},...,\alpha^{(i)}_{n_i}\}$, such that $\|\alpha^{(i)}\|^2 = 1$ and 
\begin{align}
    \nu_c(|P^{(i)}|^2/n_i) > \epsilon,
\end{align}
where $P^{(i)}(z) = \displaystyle\sum_{j=0}^{n_{i}}\alpha_{j}^{(i)}z^j$.
Note that $|P^{(i)}(z)|^2/n_i \leq 1$ for any $z\in S^1$. Let 
\begin{align}
    A^{(i)} = \{z\in S^1 \Big{|} |P^{(i)}(z)|^2 >= \frac{\epsilon}{2}n_i\}.
\end{align}
Then $\nu_s(A^{(i)}) \geq \frac{\epsilon}{2}$.

Now we pick any $L>4/\epsilon$. Lemma \ref{lem:seq} implies that we can find $\xi_1,...,\xi_L\in S^1$, $\xi_{i}\neq\xi_j$ for $i\neq j$, and a subsequence $A^{(i_j)}$ of $A^{(i)}$, such that $\xi_{k}\in A^{(i_j)}$ for any $k,j$. Then for any $i_j$, set $N = n_{i_j}$ and $\alpha = \alpha^{(i_j)}$,
\begin{align}
    \sum_{k=1}^L\frac{1}{N}|\sum_{i=0}^N\alpha_i\xi_{k}^i|^2 \geq \sum_{k=1}^L\frac{1}{N}\frac{\epsilon}{2}N \geq 2
\end{align}
This contradicts with lemma \ref{lem:nu_d1} by setting $c_1=c_2=...=c_L = 1$.
\end{proof}
Now we derive Eq.\eqref{eq:moments} from proposition \ref{prop:f_c}.
\begin{corollary}\label{cor:c moments}
Let $\nu_c$ be a continuous finite measure on $S^1$. Then for any $\epsilon > 0$,
\begin{align}
    \lim_{n\rightarrow\infty}\frac{\#\{0<k<n : |\int_{S^1}z^k d\nu_c(z)|>\epsilon\}}{n} = 0
\end{align}
\end{corollary}
\begin{proof}
For any $n$, let $M_n = \#\{0<k<n : |\int_{S^1}z^k d\nu_c(z)|>\epsilon\}$.
For any $k\leq n$, pick $\beta_k\in S^1$ so that $\displaystyle\int_{S^1}\beta_kz^kd\nu_c(z) > 0$. Let $\alpha_k = \beta_k/\sqrt{n}$. Then $\|\alpha\| = 1$ and 
\begin{align}
    \int_{S^1}|\sum_{i=0}^n\alpha_iz^i|^2d\nu_c(z) \geq \frac{1}{n}|\int_{S^1}\sum_{i=0}^n\beta_iz^id\nu_c(z)|^2\nonumber \geq\frac{M_n^2\epsilon^2}{n}
\end{align}
Then proposition \ref{prop:f_c} implies that 
\begin{align}
\lim_{n\rightarrow\infty}\frac{M_n^2\epsilon^2}{n^2}= 0.
\end{align}
Hence 
\begin{align}
\lim_{n\rightarrow\infty}\frac{M_n}{n} = 0.
\end{align}
\end{proof}

\YZ{
\section{Proof of proposition \ref{prop:cont K = disc K}}
Proposition 1:
Assume that the curve $\mathcal{K}^s: [0,\infty)\to\mathcal{H}_f^{\text{cont}}$ is continuous in $s$. Let $q$ be an eigenfrequency of the discrete-time operator $\mathcal{K}^{\Delta t}$, i.e. there exists a time-series $h\in \mathcal{H}_f \hookrightarrow \mathcal{H}^{\text{cont}}_f$ so that $\mathcal{K}^{\Delta t}h = e^{iq}h$. Then there exists at least an integer $k$, and $0\neq h_k\in\mathcal{H}^{\text{cont}}_f$, so that
\begin{align}
    \mathcal{K}^sh_k = e^{i\frac{q+2k\pi}{\Delta t}s}h_k
\end{align}
for all $s\geq 0$.
\begin{proof}[Proof of proposition \ref{prop:cont K = disc K}]
Consider $g_s\in\mathcal{H}^{\text{cont}}_f$ where $g_s = e^{-is\frac{q}{\Delta t}}\mathcal{K}^sh$. Since $\mathcal{K}^{\Delta t}h = e^{iq}h$, $g_{\Delta t} = g_0 = h$. It is easy to show that $\|g_s\| = \|h\|$ for any $s\geq 0$. So we have a closed loop in $\mathcal{H}^{\text{cont}}_f$:
\begin{align}
    g: [0,\Delta t]\rightarrow \mathcal{H}^{\text{cont}}_f.
\end{align}
Now we do Fourier decomposition to this circle, i.e. for every integer $k$, we define
\begin{align}
    h_k = \frac{1}{\sqrt{\Delta t}}\int_{0}^{\Delta t}e^{-i\frac{2k\pi}{\Delta t}s}g_sds.
\end{align}
Parseval's theorem implies that 
\begin{align}
    \Delta t\|h\|^2 = \int_{0}^{\Delta t}\|g_s\|^2 ds = \sum_{k=-\infty}^{+\infty}\|h_k\|^2.
\end{align}
Therefore there exists an integer $l$, such that $h_l\neq 0$. Now for any $t\geq 0$, 
\begin{align}
    \mathcal{K}^sh_l =& \frac{1}{\sqrt{\Delta t}}\mathcal{K}^t\int_0^{\Delta t}e^{-i\frac{2l\pi}{\Delta t}s}e^{-is\frac{q}{\Delta t}}\mathcal{K}^shds\nonumber\\
    =&\frac{1}{\sqrt{\Delta t}}\int_{0}^{\Delta t}e^{-i\frac{2l\pi + q}{\Delta t}}s\mathcal{K}^{t+s}hds\nonumber \\
    =& e^{i\frac{2l\pi + q}{\Delta t}}t\frac{1}{\sqrt{\Delta t}}\int_{0}^{\Delta t}e^{-i\frac{2l\pi + q}{\Delta t}(t+s)}\mathcal{K}^{t+s}hds\nonumber \\
    =& e^{i\frac{2l\pi + q}{\Delta t}}t \frac{1}{\sqrt{\Delta t}}\int_{0}^{\Delta t}e^{-i\frac{2l\pi}{\Delta t}(t+s)}g_{t+s}ds\nonumber \\
    =& e^{i\frac{2l\pi + q}{\Delta t}}t \frac{1}{\sqrt{\Delta t}}\int_{t}^{t + \Delta t}e^{-i\frac{2l\pi}{\Delta t}s}g_sds\nonumber \\
    =& e^{i\frac{2l\pi+q}{\Delta t}t}h_l.
\end{align}
In other words, $h_l$ is an eigen-vector of the continuous-time operator $\mathcal{K}^s$ for any $s\geq 0$.
\end{proof}
}
\appendix

\providecommand{\noopsort}[1]{}\providecommand{\singleletter}[1]{#1}%

\end{document}